\documentclass[11pt,a4paper,reqno]{amsart}
\usepackage{pdfsync}
\usepackage{mathptmx} 
\usepackage[scaled=0.90]{helvet} 
\usepackage{courier}

\normalfont
\usepackage[T1]{fontenc}
\usepackage{enumerate}

\usepackage[applemac]{inputenc}
\usepackage[T1]{fontenc}

\usepackage{microtype}  
\usepackage[shortcuts]{extdash}  
  
\usepackage{amssymb,amsrefs,amscd}
\usepackage{marvosym}
\usepackage{graphicx}
\usepackage{epstopdf}
\usepackage{lscape}
\usepackage{color}
\usepackage{verbatim}

\usepackage[curve,tips]{xypic}
\SelectTips{eu}{11}
\UseTips

\renewcommand{\epsilon}{\varepsilon}

\newtheorem{theorem}{Theorem}[section]
\newtheorem{proposition}[theorem]{Proposition}
\newtheorem{corollary}[theorem]{Corollary}
\newtheorem{lemma}[theorem]{Lemma}

\newtheorem{question}[theorem]{Question}

\newtheorem*{Main Theorem}{Main Theorem}

\theoremstyle{definition}
\newtheorem{example}[theorem]{Example}

\newtheorem{definition}[theorem]{Definition}

\theoremstyle{remark}
\newtheorem{remark}[theorem]{Remark}

\DeclareMathOperator{\cd}{cd}
\DeclareMathOperator{\gd}{gd}

\newcommand{\cll}{{\scriptstyle\bf L}}
\newcommand{\clh}{{\scriptstyle\bf H}}

\newcommand{\lhf}{\cll\clh\mathfrak F}

\newcommand{\HFF}{\clh^{\mathfrak F}{\mathfrak F}}
\newcommand{\LHFF}{\cll\clh^{\mathfrak F}{\mathfrak F}}
\newcommand{\hff}[1]{\clh^{\mathfrak F}_{#1}{\mathfrak F}}

\newcommand{\mF}{\mathfrak {F}}
\newcommand{\mX}{\mathfrak {X}}
\newcommand{\mH}{\mathfrak {H}}
\newcommand{\mT}{\mathfrak {T}}
\newcommand{\hix}[1]{\mbox{${\scriptstyle\bf H}^{\mathfrak F}_{#1}\mathfrak X$}}
\newcommand{\lhff}{{\scriptstyle\mathrm{L}}\mbox{${\scriptstyle\bf H}^{\mathfrak F}\mathfrak F$}}

\newcommand{\FF}{\mathfrak F}

\newcommand{\Q}{\mathbb Q}
\newcommand{\Z}{\mathbb Z}

\newcommand{\R}{\mathbb R}

\newcommand{\fpinfty}{{\FP}_{\infty}}

\newcommand{\FP}{\operatorname{FP}}

\newcommand{\UUFP}{\underline{\underline{\operatorname{FP}}}}
\newcommand{\UFP}{\underline{\operatorname{FP}}}

\newcommand{\EXT}{\operatorname{Ext}}
\newcommand{\Ho}{\operatorname{H}}

\newcommand{\cohom}[3]{H^{{\raise1pt\hbox{$\scriptstyle#1$}}}(#2\>\!,#3)}
\newcommand{\tatecohom}[3]%
  {\widehat H^{{\raise1pt\hbox{$\scriptstyle#1$}}}(#2\>\!,#3)}

\newcommand{\Cohom}[3]%
  {H^{{\raise1pt\hbox{$\scriptstyle#1$}}}\big(#2\>\!,#3\big)}
\newcommand{\Tatecohom}[3]%
  {\widehat H^{{\raise1pt\hbox{$\scriptstyle#1$}}}\big(#2\>\!,#3\big)}

\newcommand{\homol}[3]{H_{{\lower1pt\hbox{$\scriptstyle#1$}}}(#2\>\!,#3)}
\newcommand{\homolog}[2]{H_{{\lower1pt\hbox{$\scriptstyle#1$}}}(#2)}

\newcommand{\IND}{\operatorname{Ind}}

\newcommand{\RES}{\operatorname{Res}}

\newcommand{\colim}{\varinjlim}

\newcommand{\mono}{\rightarrowtail}
\newcommand{\epi}{\twoheadrightarrow}

\newcommand{\eg}{{\underline EG}}

\newcommand{\uueg}{{\underline{\underline E}}G}
\newcommand{\EFG}{E_{\frak F}G}

\newcommand{\Hom}{\operatorname{Hom}}

\newcommand{\Ab}{\mathfrak{Ab}}

\newcommand{\OFG}{\mathcal O_{\mathcal F}G}
\newcommand{\of}[1]{\mathcal O_{\mathfrak F}{#1}}

\newcommand{\frakF}{\mathfrak{F}}
\newcommand{\frakG}{\mathfrak{G}}

\catcode`\@=11

\newcommand{\calO}{\mathcal O}
\newcommand{\ModOFG}{\mathop{{\operator@font
Mod\text{-}}\calO_{\frakF}G}}
\newcommand{\OFGMod}{\mathop{\calO_{\frakF}G\text{-}{\operator@font
Mod}}}
\newcommand{\ModOGG}{\mathop{{\operator@font
Mod\text{-}}\calO_{\frakG}G}}
\newcommand{\OGGMod}{\mathop{\calO_{\frakG}G\text{-}{\operator@font
Mod}}}

\newcommand{\OfG}{\OC{G_{1}}{\frakF_{1}}}
\newcommand{\OgG}{\OC{G_{2}}{\frakF_{2}}}

\newcommand{\Fall}{\frakF_{\operator@font all}}
\newcommand{\Ffin}{\frakF_{\operator@font fin}}
\newcommand{\Fvc}{\frakF_{\operator@font vc}}
\newcommand{\Fic}{\frakF_{\operator@font ic}}
\newcommand{\Ffg}{\frakF_{\operator@font fg}}
\newcommand{\Fpc}{\frakF_{\operator@font pc}}
\newcommand{\Fab}{\frakF_{\operator@font ab}}
\newcommand{\Fvpc}{\frakF_{\operator@font vpc}}
\newcommand{\Fvab}{\frakF_{\operator@font vab}}

\catcode`\@=12

\newcommand{\OC}[2]{\mathop{\mathcal{O}_{#2}#1}\nolimits}
\renewcommand{\OFG}{\OC{G}{\frakF}}

\DeclareMathOperator{\PD}{pd}

\DeclareMathOperator{\id}{id}

\renewcommand{\coprod}%
{\mathop{\rotatebox[origin=c]{180}{$\displaystyle\prod$}}\limits}

\renewcommand{\:}{\mathord{:}\hspace*{0.8ex plus .25ex minus .1ex}}

\title[Complete Bredon cohomology]{Complete Bredon cohomology and its applications to hierarchically defined groups}

\author{Brita  E.~A.~Nucinkis}
\address{ Department of Mathematics,
Royal Holloway, University of London, 
Egham, TW20 0EX }
\email{Brita.Nucinkis@rhul.ac.uk}

\author{Nansen Petrosyan}
\address{School of Mathematics, University of Southampton, Southampton SO17 1BJ}
\email{N.Petrosyan@soton.ac.uk}
\date{\today} 

\begin{document}

\begin{abstract}  By considering the  Bredon analogue of  complete cohomology of a group, we show that every  group in the class $\LHFF$ of type Bredon-$\fpinfty$ admits a finite dimensional model for $\EFG$. 

We also show that abelian-by-infinite cyclic groups admit a $3$-dimensional model for the classifying space for the family of virtually nilpotent subgroups.  This allows us to prove that for  $\mF,$ the class of virtually cyclic groups, the class of $\LHFF$-groups contains all locally virtually soluble groups and  all   linear groups over $\mathbb C$  of integral characteristic.
\end{abstract}

\maketitle

\thispagestyle{empty}


\section{Introduction}

Classifying spaces with isotropy in a family have been the subject of intensive research, with a large proportion focussing on
 $\eg$, the classifying space with finite isotropy~\cites{lueckbook, lueck, luecksurvey}. Classes of groups admitting a finite dimensional model for $\eg$ abound, such as elementary amenable groups  of finite Hirsch length ~\cites{fn, kmn},  hyperbolic groups \cite{MS}, mapping class groups \cite{luecksurvey} and $Out(F_n)$ \cite{V}. Finding  manageable models for $\uueg,$ the classifying space for virtually cyclic isotropy, has been
shown to be much more elusive.   So far manageable models have
 been found for crystallographic groups \cite{lafontortiz},
polycyclic-by-finite groups~\cite{lueckweiermann}, hyperbolic groups \cite{jpl}, certain HNN-extensions  \cite{fluch-11}, elementary amenable groups of finite Hirsch length \cite{fln, DP13, DegPet2} and groups acting isometrically with discrete orbits on separable complete  CAT(0)-spaces \cite{luck3, DegPet4}.

Let $\frak F$ be a family of subgroups of a given group and denote by $\EFG$ the classifying space with isotropy in $\frak F.$ In this note we propose a method to decide whether a group has a finite dimensional model for $\EFG$ without actually providing a bound. This is closely related to Kropholler's Theorem that a torsion-free group in $\lhf$ of type $\fpinfty$ has finite integral cohomological dimension \cite{kropholler}. To do this we consider groups belonging to the class $\LHFF$, a class recently considered in \cite{DPT}:

 Let $\mF$ be a  class of groups closed under taking subgroups. Let $G$ be a group and set $\mF\cap G=\{H\leq G \; | \;  \mbox{$H$ is isomorphic to a subgroup in $\mF$}\}.$   Let $\mX$ be a class of groups. Then \hix{} is defined as the smallest class of groups containing the class
$\mX$ with the property that if a group $G$ acts cellularly on a
finite dimensional CW-complex $X$ with all isotropy
subgroups in \hix{}, and such that for each subgroup $F\in \mF\cap G$ the fixed point set
$X^F$ is contractible, then $G$ is in \hix{}. The class ${{\scriptstyle\mathrm{L}\bf H}}^\mathfrak{F}\mX$ is defined to be the class of groups that are locally \hix{}-groups.\\
In this definition and throughout  the paper, we always assume that a cellular action of group on a CW-complex is  admissible. That is, if  an element of group stabilises a cell, then it fixes it pointwise. \\
We generalise complete cohomology of a group to the Bredon setting and verify that some of the main results hold in this new context. This allows us to establish:

\bigskip\noindent {\bf Theorem A.} {\sl Let $G$ be group in $\LHFF$ of type Bredon-$\fpinfty$. Then $G$ admits a finite dimensional model for $\EFG.$}\\

We  consider the class ${{\scriptstyle\mathrm{L}\bf H}}^\mathfrak{F}\mX$, especially when $\mF=\mX$ is either the class of all finite groups or the class of all virtually cyclic groups.
Note, that if $\mF$ contains the trivial group only, then ${\scriptstyle\mathrm{L}}$\hix{} is exactly Kropholler's class  ${\scriptstyle\mathrm{L}}{{\scriptstyle\bf H}}\mathfrak{X}$.
If $\mF$ is the class of all finite groups, then ${\scriptstyle\mathrm{L}}{{\scriptstyle\bf H}}^{\mathfrak{F}}\mathfrak{F}$ also turns out to be quite large. It contains all  elementary amenable groups and all  linear groups over a field of arbitrary  characteristic (see \cite{DegPet4}, \cite{DPT}). It is also closed under extensions, taking subgroups, amalgamated products, HNN-extensions, and countable directed unions. Here we show that similar closure operations hold when $\mF$ is the class of virtually cyclic groups.

In \cite{DPT}, it was shown that when $\mF$ is the class of finite groups, then $\lhff$ contains all elementary amenable groups. We show that when  $\mF=\mF_{vc}$, the class of virtually cyclic groups, then ${{\scriptstyle\mathrm{L}\bf H}}^{\Fvc}\Fvc$ contains all locally virtually soluble groups.
We also show that any countable subgroup of a general linear group $\mathrm{GL}_n(\mathbb C)$ of integral characteristic lies in ${{\scriptstyle\bf H}}^{\Fvc}\Fvc$. Both of these results rely on the following:

\bigskip\noindent {\bf Theorem B.}  {\sl Let $G$ be a semi-direct product $A\rtimes \Z$  where $A$ is a countable abelian group. Define $\mH$ to be the family of all virtually nilpotent subgroups of $G$.
Then there exists a $3$-dimensional model for $E_{\mH}G$.}\\

Another consequence of Theorem B is that any semi-direct product $A\rtimes \Z$  where $A$ is a countable abelian group lies in ${{\scriptstyle\bf H}}_3^{\Fvc}\Fvc$.


\section{Background on Bredon cohomology}

In this note, a  family $\frakF$ of subgroups of a group $G$
 is closed under conjugation and taking subgroups. The families most frequently considered are the family 
$\Ffin(G)$ of all finite subgroups of $G$ and the family $\Fvc(G)$ of 
all virtually cyclic subgroups of $G$. 

For a subgroup $K \leq G$ we consider:
\begin{equation*}
    \frakF\cap K = \{ H\cap K \mid H\in \frakF\}.
\end{equation*}

Bredon cohomology has been introduced for finite groups by
Bredon~\cite{bredon-67} and later  generalised to arbitrary
groups by Lück~\cite{lueckbook}. 

The orbit category $\OFG$ is defined as follows:  objects are 
the transitive $G$-sets~$G/H$ with $H\leq G$ and $H \in \FF$; morphisms of $\OFG$ are all $G$-maps $G/H \to G/K$, where $H,K \in \FF.$ 

An $\OFG$-module, or Bredon module, is a contravariant functor $M\: \OFG\to \Ab$ from the orbit category to the category of
abelian groups. A natural transformation $f\: M\to N$ between two
$\OFG$-modules  is called a morphism of
$\OFG$-modules. 

The trivial 
$\OFG$-module is denoted by $\Z_{\frakF}.$  It is given by $\Z_\FF(G/H) = \Z$ and \\
$\Z_\FF(\varphi)=\id$ for all objects and morphisms $\varphi$ of $\OFG$.

The category of  $\OFG$-modules, denoted $\ModOFG,$   is a functor  category and therefore  inherits   properties from the 
category $\Ab$.  For example,  a sequence  $L \to M\to N$ of Bredon 
modules is exact if and only if, when evaluated at every $G/H\in \OFG,$ we 
obtain an exact sequence $L(G/H) \to M(G/H)\to N(G/H)$ of abelian groups.

Since $\Ab$ has enough projectives, so does $\ModOFG,$ and we can define homology functors 
 in $\ModOFG$ analogously to ordinary cohomology, using projective resolutions.

\medskip
There now follow the basic properties of free and projective $\OFG$-modules as described in \cite[9.16, 9.17]{lueckbook}. An $\FF$-set $\Delta$ is a collection of sets $\{\Delta_K \,|\, K \in \FF\}.$ For any two $\FF$-sets $\Delta$ and $\Omega$, an $\FF$-map is a family of maps $\{\Delta_K \to \Omega_K\,|\, K \in \FF\}$. Hence we have a forgetful functor from the category of $\OFG$-modules to the category of $\FF$-sets. One defines the free functor as the left adjoint to this forgetful functor. This  satisfies the usual universal property. 

There is a more constructive description of free Bredon-modules as follows:
Consider the right Bredon-module: $\Z[-,G/K]_\FF$
with $K\in\frakF$.  When  evaluated at $G/H$
we obtain the free abelian group $\Z[G/H,G/K]_\FF$ on the set
$[G/H,G/K]_\FF$ of $G$-maps $G/H\to G/K$. These modules are free, cf.~\cite{lueckbook}*{p.~167}, and can be viewed as the building blocks of the free right Bredon-modules.
Generally, a free module is one of the form $\Z[-,\Delta]_\FF$, where $\Delta$ is a $G$-set with isotropy in $\FF.$ Projectives are now defined to be direct summands of frees.

\medskip

Given a covariant functor $F\: \OfG \to \OgG$ between orbit categories, one can now define induction and restriction functors along $F$, see
~\cite{lueckbook}*{p.~166}:

$$\begin{array}{lccc}
\IND_{F}: & \OfG & \to &\OgG \\
& M(-) & \mapsto & M(-)\otimes_{\frakF_{1}} [--,F(-)]_{\FF_2}
\end{array}$$
and
$$\begin{array}{lccc}
\RES_{F}: & \OgG & \to &\OfG \\
& M(--) & \mapsto & M\circ F(--)
\end{array}$$

Since these functors are adjoint to each others,   $\IND_{F}$ commutes with arbitrary
colimits~\cite{mac-lane-98}*{pp.~118f.} and  preserves free and
projective Bredon modules~\cite{lueckbook}*{p.~169}.
The case of particular interest is when $F$ is given by inclusion of a subgroup of $G$.

For subgroup $K$ of $G$  we consider the following functor
$$\begin{array}{lccc}
    \iota^G_{K}\: & \OC{K}{\frakF\cap K} &\to& \OFG \\
    & K/H & \mapsto & G/H.
\end{array}$$
and denote the corresponding induction and restriction functors by $\IND_K^G$ and  $\RES_K^G$ respectively.

\begin{lemma}
    \cite{symonds-05}*{Lemma 2.9}
    \label{lem:symonds}
    Let $K$ be a subgroup of $G$. Then $\IND^G_{K}$ is an 
    exact functor.
\end{lemma}

Symmond's \cite{symonds-05} methods also yields; for a short account see also the proof of Lemma 3.5 in \cite{kmn}:

\begin{lemma}\label{indlemma}
Let $K \leq H \leq G$ be subgroups.  Then
$$\IND_K^G\Z_\FF \cong \Z[-, G/K]_\FF,$$
and
$$\IND_H^G\Z[-,H/K]_{\FF\cap H} \cong \Z[-,G/K]_\FF.$$
\end{lemma}

\medskip

The Bredon cohomological dimension $\cd_{\frakF}G$ of a group
$G$ with respect to the family $\frakF$ of subgroups is the projective
dimension $\PD_{\frakF}\Z_{\frakF}$ of the trivial
$\OFG$-module $\Z_\FF$. 
The cellular chain complex of a model for
$E_{\frakF}G$ yields a free resolution of the trivial
$\OFG$-module $\Z_{\frakF}$~\cite{lueckbook}*{pp.~151f.}.  In
particular, this implies that for the Bredon geometric dimension
$\gd_{\frakF}G$, the minimal dimension of a model for $E_{\frakF}G$, we have 
$$\cd_\FF G \leq \gd_\FF G.$$

Furthermore, one always has:

\begin{proposition}
    \cite{lm}*{Theorem~0.1~(i)}
    \label{prop:lm}
    Let $G$ be a group. Then
    \begin{equation*}
        \gd_{\frakF}G\leq\max(3,\cd_{\frakF}G).
    \end{equation*}
\end{proposition}

\smallskip

Next, suppose  $\mT$ and $\mH$ are  families of subgroups  of a group $G$ where $\mT\subseteq \mH$. In Section 5, we will need to adapt a model for $E_{\mT}G$  to obtain a model for $E_{\mH}G$. For this we will use a general construction of  L\"{u}ck and Weiermann (see \cite[\S 2]{lueckweiermann}). We recall the basics of this construction:

Suppose that there exists an equivalence relation $\sim$ on the set $\mathcal{S}=\mH\smallsetminus \mT$ that satisfies the following properties:
\smallskip
\begin{itemize}
\item  $\forall H,K \in \mathcal{S} : H \subseteq K \Rightarrow H \sim K$;
\medskip
\item $ \forall H,K \in \mathcal{S},\forall x \in G: H \sim K \Leftrightarrow H^x \sim K^x$.
\end{itemize}
\smallskip
An equivalence relation that satisfies these properties is called a \emph{strong equivalence relation}.
Let $[H]$ be an equivalence class represented by $H \in \mathcal{S}$ and denote the set of equivalence classes by $[\mathcal{S}]$. The group $G$ acts on $[\mathcal{S}]$ via conjugation, and the stabiliser group of an equivalence class $[H]$ is
\begin{equation*}\label{eq gen normaliser} \mathrm{N}_{G}[H]=\{x \in \Gamma \ | \ H^x \sim H \}. \end{equation*}
Note that $\mathrm{N}_{G}[H]$ contains $H$ as a subgroup. Let $\mathcal{I}$ be a complete set of representatives $[H]$ of the orbits of the conjugation action of $G$ on $[\mathcal{S}]$. 
Define for each $[H] \in \mathcal{I}$ the  family
\begin{equation*} \label{eq: special family}{\mT}[H]=\{ K \leq \mathrm{N}_{G}[H] \ | K \in \mathcal{S}, K \sim H\} \cup \Big(\mathrm{N}_{G}[H] \cap \mT\Big) \end{equation*} of subgroups of $\mathrm{N}_{G}[H]$.
\begin{proposition}[{L\"{u}ck-Weiermann, \cite[2.5]{lueckweiermann}}] \label{LW} Let $\mT \subseteq \mH$ be two families of subgroups of a group $G$ such that $S=\mH\smallsetminus \mT$ is equipped with a strong equivalence relation. Denote the set of equivalence classes by $[\mathcal{S}]$ and let $\mathcal{I}$ be a complete set of representatives $[H]$ of the orbits of the conjugation action of $G$ on $[\mathcal{S}]$. If there exists a natural number $d$ such that  $\mathrm{gd}_{\mT\cap \mathrm{N}_{G}[H]}(\mathrm{N}_{G}[H]) \leq d-1$ and
$\mathrm{gd}_{\mT[H]}(\mathrm{N}_{G}[H]) \leq d$ for each $[H] \in \mathcal{I}$,
and such that $\mathrm{gd}_{\mT}(G) \leq d$, then $\mathrm{gd}_{\mH}(G) \leq d$.
\end{proposition}

\section{Complete Bredon cohomology}

Since $\ModOFG$ is an abelian category, we can just follow the approaches of Mislin \cite{mislin94} and Benson-Carlson \cite{BC}. We will, however, include the main steps of the construction. 
We will begin by describing the Satellite construction due to Mislin \cite{mislin94}. The methods used there can be carried over  to the Bredon-setting by applying \cite[XII.7-8.]{McL}.

Let $M$ be an $\OFG$-module and denote by $FM$ the free $\OFG$-module on the underlying $\FF$-set of $M$. Let $\Omega M =ker(FM\epi M),$ and inductively $\Omega^nM=\Omega(\Omega^{n-1}M).$  Let $T$ be an additive functor from $\ModOFG$ to the category of abelian groups. Then the left satellite of $T$ is defined as
$$S^{-1}T(M)=ker(T(\Omega M) \to T(FM)).$$
Furthermore, $S^{-n}T(M)=S^{-1}(S^{-n+1}T(M)),$ and the family $\{S^{-n}\,|\, n \geq 0\}$ forms a connected sequence of  functors where $S^{-n}T(P)=0$ for all projective $\OFG$-modules $P$ and  $n\geq 1.$
Following the approach in \cite{mislin94} further, we call a connected sequence of additive functors $T^*=\{T^n\, |\, n\in \Z\}$ from $\ModOFG$ to the category of abelian groups a $(-\infty,+\infty)$-Bredon-cohomological functor, if for every short exact sequence $M' \mono M \epi M''$ of $\OFG$-modules the associated sequence
$$ \cdots \to T^nM' \to T^nM \to T^n M'' \to T^{n+1}M' \to \cdots$$
is exact. Obviously, Bredon-cohomology $\Ho^*_\FF(G,-)$ is such a functor with the convention that $\Ho^n_\FF(G,-)=0$
whenever $n<0.$

\begin{definition}
 A $(-\infty,+\infty)$-Bredon-cohomological functor $T^*=\{T^n\, |\, n\in \Z\}$ is called $P$-complete if $T^n(P)=0$ for all $n\in \Z$ and every projective $\OFG$-module $P$. 

 A morphism $\varphi^*: U^* \to V^*$ of $(-\infty,+\infty)$-Bredon-cohomological functors is called a $P$-completion, if $V^*$ is $P$-complete and  if every morphism $U^* \to T^*$ into a $P$-complete  $(-\infty,+\infty)$-Bredon-cohomological functor $T^*$ factors uniquely through $\varphi^*: U^* \to V^*$. 
\end{definition}

The following theorem is now the exact analogue to \cite[Theorem 2.2]{mislin94}. 

\begin{theorem} Every $(-\infty,+\infty)$-Bredon-cohomological functor $T^*$ admits a unique $P$-completion $\widehat T^*$ given by   $$\widehat{T}^j(M)= \colim_{k\geq 0} S^{-k} T^{j+k}(M)$$ for any $M\in  \ModOFG$.\qed
\end{theorem}
In particular, we have, for every $\OFG$-module $M$ that
$$\widehat\EXT^j_\FF(M,-)= \colim_{k\geq 0} S^{-k} \EXT^{j+k}_\FF(M,-).$$

We have immediately:

\begin{lemma} Let $M$ and $N$ be  $\OFG$-modules. If either of these has finite projective dimension, then 
$$\widehat\EXT^*_\FF(M,N)=0.$$
\end{lemma}

We can also mimic Benson and Carlson's approach \cite{BC}.  For any two $\OFG$-modules we denote by $[M,N]_\FF$ the quotient of $\Hom_{\FF}(M,N)$ by the subgroup of those homomorphisms factoring through a projective module. Then it follows that there is a homomorphism $[M,N]_\FF \to [\Omega M,\Omega N]_\FF$ and it can be shown analogously to \cite[Theorem 4.4]{mislin94} that
$$\widehat\EXT^n_\FF(M,M)= \colim_{k,k+n\geq 0}[\Omega^{k+n}M,\Omega^kN]_{\FF}.$$

This now allows us to deduce  the following Lemma, which is an analogue to \cite[4.2]{kropholler}. 

\begin{lemma}\label{tateext0lemma}
$\widehat\EXT^0_\FF(M,M)=0$ if and only if $M$ has finite projective dimension. In particular,
$$\widehat\Ho^0_\FF(G,\Z_\FF)=0 \iff \cd_\FF G < \infty.$$
\end{lemma}

\section{Proof of Theorem A}

\noindent The proof of Theorem A is analogous to the proof of the main result in \cite{kropholler}.  We begin by recording two easy lemmas, which have their analogues in \cite[3.1]{kropholler} and \cite[4.1]{kropholler} respectively.

\begin{lemma}\label{tateexteslemma}
Let
$$0\to M_n\to M_{n-1} \to ...\to M_1\to M_0 \to L\to 0$$ be an exact sequence of $\OFG$-modules and $i$ be an integer such that $\Ho^i_\FF(G,L) \neq 0.$
Then there exists an integer $0\leq j \leq n-1$ such that $\Ho^{j+i}_\FF(G,M_j) \neq 0.$.
\end{lemma}

\begin{proof} This is an easy dimension shifting argument.
\end{proof}

\begin{lemma}\label{tatects}
Let $G$ be a group such that $\Ho_\FF^k(G,-)$ commutes with direct limits for infinitely many $k$,  then $\widehat\Ho^k_\FF(G,-)$ commutes for all $k \in \Z.$
\end{lemma}

\begin{proof} This follows from the fact that direct limits commute with each other.
\end{proof}

The proof of Theorem A now relies on the fact that one can hierarchically decompose the class $\HFF$ in exactly the same way as Kropholler's decomposition, see \cite{kropholler, DPT}:  
\begin{itemize}
\item $\hff0=\FF;$
\item For an ordinal $\alpha >0$, we let $\hff\alpha$ be the class of groups acting cellularly on a finite dimensional complex $X$ such that each stabiliser subgroup lies in $\hff\beta$ for some $\beta<\alpha$ and such that $X^K$ is contractible for all $K \in \FF.$
\end{itemize}
A group $G$ now lies in $\HFF$ if and only if it lies in some $\hff\alpha$ for some ordinal $\alpha.$ 

\noindent In particular, $\HFF$ is subgroup closed.

\begin{lemma}\label{limlemma} Let $G$ be a group and $G_\lambda$, $\lambda \in \Lambda$ its finitely generated subgroups. Then we have the following isomorphism:
$$\colim_{\lambda \in \Lambda} \Z[-,G/G_\lambda]_{\FF\cap G_\lambda} \cong \Z_\FF.$$   
\end{lemma}

\begin{proof}
This follows directly from Lemma \ref{indlemma}.
\end{proof}

\begin{theorem}\label{big}
Let $G$ be a group in $\LHFF$ and suppose that $\widehat\Ho_\FF^*(G,-)$ commutes with direct limits. Then $\cd_\FF G < \infty.$
\end{theorem}

\begin{proof} 
We prove this by contradiction and suppose that $\cd_\FF G = \infty.$ Hence, by Lemma \ref{tateext0lemma}, we have that $\widehat\Ho^0_\FF(G, \Z_\FF) \neq 0.$ We claim that then there exists a group $H \in \FF$ and an integer $i \geq 0$ such that $\widehat\Ho^i_\FF(G, \IND_H^G\Z_{\FF \cap H}) \neq 0.$ By Lemma \ref{indlemma}, we have $\IND_H^G\Z_{\FF \cap H} \cong \Z[-,G/H]$, which is projective, giving us the desired contradiction. 

It now remains to prove the claim:  Let $\mathcal S$ be the set of ordinals $\beta$ such there exists a $i \geq 0$ and $H \leq G$ lying in $\hff\beta$ and such that $\Ho^i_\FF(G, \IND_H^G\Z_{\FF\cap H} )\neq 0.$ If we can prove that $0\in \mathcal{S}$, we are done.
 
 \noindent (1) We show that $\mathcal S$ is not empty:  Let $\{G_\lambda \,|\, \lambda\in \Lambda\}$ be the family of all finitely generated subgroups of $G$. Hence, applying Lemma \ref{limlemma} and the fact that  $\widehat\Ho_\FF^*(G,-)$ commutes with direct limits, we get
 $$\widehat\Ho_\FF^0(G,\Z_\FF)\cong \widehat\Ho_\FF^0(G,\colim_{\lambda \in \Lambda} \Z[-,G/G_\lambda]_{\FF\cap G_\lambda})\cong \colim_{\lambda \in \Lambda}\widehat\Ho_\FF^0(G, \Z[-,G/G_\lambda]_{\FF\cap G_\lambda}).$$
Since $\widehat\Ho_\FF^0(G,\Z_\FF)\neq 0,$ there exists a finitely generated subgroup $G_\lambda$ such that, see also Lemma \ref{indlemma},
 $$\widehat\Ho_\FF^0(G, \Z[-,G/G_\lambda]_{\FF\cap G_\lambda})\cong \widehat\Ho_\FF^0(G, \IND_{G_\lambda}^G\Z_{\FF\cap G_\lambda})\neq 0.$$ Since $G\in \LHFF,$ and $\HFF$ is subgroup closed, $G_\lambda \in \HFF$ and in particular, there is an ordinal $\beta$ such that $G_\lambda\in \hff\beta$. Hence $\beta\in {\mathcal S}.$
 
 \medskip\noindent (2) We now show that, if $0\neq \beta \in {\mathcal S}$, then there is an ordinal $\gamma<\beta$ such that $\gamma \in {\mathcal S}:$  Let $0\neq \beta \in {\mathcal S}.$ Then there is a $H \in G$ and $i \geq 0$ such that $H\in \hff\beta$ and
 $$\widehat\Ho^i_\FF(G, \IND_H^G\Z_{\FF\cap H}) \neq 0. $$
 Hence $H$ acts cellularly on a finite dimensional contractible space $X$ such that each isotropy group lies in some $\hff\gamma$ for  $\gamma < \beta$ and such that $X^K$ is contractible if $K \in \FF.$  Hence we have an exact sequence of free $\of{H}$-modules:
 $$0 \to C_n(X^{(-)}) \to C_{n-1}(X^{(-)}) \to ...\ C_1(X^{(-)}) \to C_0(X^{(-)} )\to \Z_{\FF\cap H} \to 0.$$
 Each 
 $$C_k(X^{(-)})\cong \Z[-, \bigoplus_{\sigma_k\in \Delta_k}H/H_{\sigma_k}],$$
 where $\Delta_k$ is the set of orbit representatives for the $k$-cells of $X$.
 Furthermore, by Lemma \ref{indlemma}, upon induction, we obtain an exact sequence of $\OFG$-modules as follows:
 $$0 \to \bigoplus_{\sigma_n\in \Delta_n}\IND_{H_{\sigma_n}}^G\Z_{\FF\cap H_{\sigma_n}} \to ...\to \bigoplus_{\sigma_1\in \Delta_1}\IND_{H_{\sigma_1}}^G\Z_{\FF\cap H_{\sigma_1}} \to \bigoplus_{\sigma_0\in \Delta_0}\IND_{H_{\sigma_0}}^G\Z_{\FF\cap H_{\sigma_0}} \to \IND_H^G\Z_{\FF\cap H} \to 0.$$
 
 Now, by Lemma \ref{tateexteslemma}, there is a $k \geq 0$ such that 
 $$\widehat\Ho^{j+k}_{\FF}(G, \bigoplus_{\sigma_k\in \Delta_k}\IND_{H_{\sigma_k}}^G\Z_{\FF\cap H_{\sigma_k}}) \neq 0.$$
 Since $\widehat\Ho^{j+k}_{\FF}(G,-)$ commutes, in particular,  with direct sums, there is a $\sigma_k \in \Delta_k$ such that
 $$\widehat\Ho^{j+k}_{\FF}(G,\IND_{H_{\sigma_k}}^G\Z_{\FF\cap H_{\sigma_k}}) \neq 0,$$
 thus proving the claim.

\end{proof}

\begin{corollary}
Let $G$ be a group in $\HFF$ and suppose that $\widehat\Ho_\FF^*(G,-)$ commutes with direct sums. Then $\cd_\FF G < \infty.$
\end{corollary}

\begin{proof}
The proof is analogous to the proof of Theorem \ref{big}. To show that $\mathcal S$ is not empty, we can use the fact that $G \in \hff\beta$ for some $\beta.$ Then follow step (2) as above.
\end{proof}

\medskip\noindent Theorem A now follows directly from Theorem \ref{big}, as, for groups of type Bredon-$\fpinfty$ it follows that $\widehat\Ho^*_\FF(G,-)$ commutes with direct limits, see Lemma \ref{tatects} and \cite[Theorem 5.3]{martineznucinkis11}.

\section{Some properties of $\LHFF$}

We consider containment and closure properties of the class ${{\scriptstyle\mathrm{L}\bf H}}^\mathfrak{F}\mX$ especially when $\mF$ either the class of finite groups or the class of virtually cyclic groups.\\

Let $A$ be an abelian group and $\Z=\langle t\rangle$. Consider the semi-direct product $G= A\rtimes \Z$ with $t$ acting on $A$ by conjugation. To shorten the notation, wherever necessary, we will identify $A$ with its image in $G$.  Fix an arbitrary integer $k> 0$. For each integer $i\geq 0$, we define the subgroups  $P^k_i$ of $A$ inductively as follows:
\begin{itemize}
\item $P^k_0=\langle 1\rangle,$
\smallskip
\item $P^k_{i+1}=\{ x\in A \; |   \; t^k(x)x^{-1}\in P^k_i \}$ for $i\geq 0$.
\end{itemize}
An easy induction on $i$ shows that each $P^k_i$ is a normal subgroup of $G$. We set  $P^k=\cup_{i\geq 0}P^k_i$.  Note that $P^k$ is also a normal subgroup of $G$ and it has the property that if $t^k(x)x^{-1}\in P^k$ and $x\in A$  then $x\in P^k$. In fact, $P^k$ can be defined as the smallest subgroup of $G$ with this property.

\begin{lemma}\label{nesting} Let $a\in A$. For each $i\geq 0$, consider the subgroup $G^k_i=\langle P^k_i, (a,t^k)\rangle$ of $G$. Then $P^k_i= G^k_i\cap A$ and $G^k_i$ is  nilpotent of nilpotency class at most $i+1$. \end{lemma}
\begin{proof} For the first part one only needs to check that $G^k_i\cap A$ is in $P^k_i$ as the reverse inclusion is trivially satisfied. But this follows immediately from the fact that $G^k_i \cong P_i^k \rtimes \langle (a,t^k)\rangle.$ 

For the second claim, note that $[G^k_i, G^k_i]$ lies in $A$. Let $0\leq m\leq i$. The only possibly nontrivial $m$-fold commutators starting with an element $x\in P_i^k$ are of the form 
$$y_m=[(a_1, t^{kn_1}), [ (a_2, t^{kn_2}), \dots [(a_m, t^{kn_m}), x]\dots]]$$ for $a_1, \dots, a_m \in P^k_i$ where we denote $y_0=x$. We claim that $y_m$ is in $P^k_{i-m}$. Assuming the claim, we have that $y_i$ is trivial and hence $G^k_i$ is nilpotent of nilpotency class at most $i+1$.

To prove the claim we use induction on $m$. The case $m=0$ is trivially satisfied. Now, suppose $m>0$. Then, by induction,  the $(m-1)$-fold commutator $$z=[ (a_2, t^{kn_2}), \dots [(a_m, t^{kn_m}), x]\dots]\in P^k_{i-m+1}.$$ But then
$$y_m=[(a_1, t^{kn_1}), z]=t^{kn_1}(z)z^{-1}\in P^k_{i-m}$$ because $z\in P^k_{i-m+1}$. This finishes the claim.
\end{proof}

\begin{lemma}\label{converse}  For a given integer $i> 0$, let $N$ be a nilpotent subgroup of $G$ of nilpotency class $i$, which is not contained in $A$. Then  $N=\langle B, (a,t^k)\rangle$ where $B=P^k_i\cap N$ for some $a\in A$ and $k>0$.  In particular, $N$ is contained in  $G^k_i=\langle P^k_i, (a,t^k)\rangle$.\end{lemma}
\begin{proof} Clearly, $N=\langle B, (a,t^k)\rangle$ where $B=A\cap N$ for some $a\in A$ and $k>0$.  It is left to show that $B\leq P^k_i$. Let $0\leq m\leq i$ and consider $(i-m)$-fold commutator
$$y_{(i-m)}=[(a, t^{k}), [ (a, t^{k}), \dots [(a, t^{k}), x]\dots]]$$ where we denote $y_0=x\in B$. We  will prove by induction that  $y_{(i-m)}\in P^k_m$. Since $N$ has nilpotency class $i$, $y_{i}=1\in P^k_0$. So, assume $m>0$. Consider $z=[(a, t^{k}), y_{(i-m)}]$. By induction, $z\in P^k_{m-1}$. But $z=t^{k}(y_{(i-m)})y_{(i-m)}^{-1}$. So, by the definition of $P^k_m$,  we have $y_{(i-m)}\in P^k_m$.  

Now, taking $m=i$, gives us that each $x\in B$ lies in $P^k_i$.
\end{proof}

\begin{proposition}\label{P} Define $P=\cup_{k> 0}P^k$ in $A$. Then 
\begin{itemize}
\item[(a)] $P$ is a normal subgroup of $G$.
\smallskip
\item[(b)] $P$ is the smallest subgroup of $G$ defined by the property that if $t^k(x)x^{-1}\in P$ for some $k>0$ and $x\in A$, then $x\in P$. 
\smallskip
\item [(c)] Let $N=\langle B, (a, t^l)\rangle$ where $B\leq P$, $a\in A$ and $l\geq 1$. Then $N$ is locally virtually nilpotent.
\smallskip
\item[(d)] Let $N$ be a locally nilpotent subgroup of $G$ not contained in $A$. Then $N\cap A$ is contained in $P$.
\end{itemize} \end{proposition}
\begin{proof} {\sl (a).} Given any integers $k_1, k_2> 0$ such that $k_1$ divides $k_2$, it follows that $P_{k_1}\subseteq P_{k_2}$. This shows that the set $P$ is a subgroup of $A$.  Since each $P_k$ is a normal subgroup of $G$, their union $P$ is also a normal in $G$.\\
{\sl (b).}  Let $P'$ be the smallest subgroup of $G$ defined by the property stated in (b); denote this property by ($\ast$). Note that $P=\cup_{i\geq 0}P_i$ where the subgroups  $P_i$ are defined inductively by:
\begin{itemize}
\item $P_0=\langle 1\rangle,$
\smallskip
\item $P_{i+1}=\{ x\in A \; |  \; \exists k>0,  t^k(x)x^{-1}\in P_i \}$ for $i\geq 0$.
\end{itemize}
An easy induction on $i$ shows that each $P_i$ is a subgroup of $P'$. Hence, $P\leq P'$. But since $P$ has the property ($\ast$) and  $P'$ is the smallest subgroup of $G$ with the property  ($\ast$), we deduce  that $P=P'$.\\ 
{\sl(c).}  Let $H= \langle b_1, \dots, b_s, (a, t^l)\rangle$, for some $b_1, \dots, b_s \in P$, $a\in A$, and $l, s\geq 1$. It suffices to show that $H$ is virtually nilpotent. Since $P=\cup_{i, k>0} P^k_i$, we conclude that  for each $j\in \{1,...,s\}$, we have $b_j\in P^{k_j}_{i_j}$ for some $i_j, k_j>0$. Set $k=\prod_{j=1}^s k_{i_j}$ and $i=\sup\{i_j \;| \; 1\leq j\leq s\}$. It follows that the group $H'= \langle b_1, \dots, b_s, (a, t^l)^k\rangle$ is a finite index subgroup of $H,$ and $H'\leq \langle P^{kl}_i, (a,t^l)^k\rangle$. So, by Lemma \ref{nesting}, $H'$ is nilpotent.\\
{\sl(d).}  This is a direct consequence of Lemma \ref{converse}.
\end{proof}

\begin{theorem}\label{cyclic ext}  Let $G$ be a semi-direct product $A\rtimes \Z$  where $A$ is a countable abelian group. Define $\mH$ to be the family of all  virtually nilpotent subgroups of $G$.
Then there exists a $3$-dimensional model for $E_{\mH}G$.
\end{theorem}
\begin{proof}   Let $\mT$ be the subfamily of  $\mH$ consisting of all countable  subgroups of $A$. We will use the construction of L\"{u}ck and Weiermann that adapts the model for $E_{\mT}G$ to a model for the larger family $\mH$.

First, we need a strong equivalence relation on the set $$\mathcal S=\mH\smallsetminus \mT=\{H\leq G \;|\; H\not\leq A \mbox{ and } H \mbox{ is virtually nilpotent}\}.$$

Let $^{\overline{\;\;\;}}:G\to G/P$ denote the quotient homomorphism. By Proposition \ref{P}, we have that if  $H\in \mH$, then $\overline{H}$ is virtually cyclic.

Now, for $H, S \in \mathcal{S}$, we say that there is a relation $H\thicksim S$ if $|\overline{H}\cap \overline{S}|=\infty$. It is not difficult to show that this indeed defines a  strong equivalence relation on the set $\mathcal{S}$. Our group $G$ acts by conjugation on the set of equivalence classes $[\mathcal{S}]$ and the stabiliser of an equivalence class $[H]$ is $$N_{G}[H]=\{x\in G \;|\; H^x\thicksim H\}.$$Note that $H\thicksim Z$ if $Z=\langle h \rangle$, $h\in H$, $h\notin A$. Hence $N_{G}[H]=N_{G}[Z].$ Clearly, $Z$ is a subgroup of $N_{G}[Z]$ and $N_{G}[Z]=\langle B, Z\rangle$ for some subgroup $B\leq A$.  But for each $b\in B$, we have $Z^b \sim Z$. Writing $h=(a,t^k)$ for some $a\in A$ and $k>0$, this implies that $\overline{b^{-1}(a,t^k)^nb}=\overline{(a,t^k)^n}$  in $G/P$ for some nonzero integer $n$.  A quick computation then shows that  $\overline{t^{kn}(b)}=\overline{b}$ in $G/P$. This means that $t^{kn}(b)b^{-1}\in P$. Then, by Proposition \ref{P}(b),  $b\in P$. Hence, by part (c) of Proposition \ref{P},  we have that every finitely generated subgroup $K$ of $N_{G}[Z]$ that contains $Z$ is virtually nilpotent. Thus $K\in \mathcal{S}$ and $K\sim Z$ and hence it is in the family 
$$\mT[H]=\{ K \leq \mathrm{N}_{G}[H] \ | K \in \mathcal{S}, K \sim H\} \cup \Big(\mathrm{N}_{G}[H] \cap \mT\Big)$$ of subgroups of $\mathrm{N}_{G}[H]$.
It follows that $N_{G}[H]$ is a countable directed union of subgroups  that are in $\mT[H]$ but are not in $\mathrm{N}_{G}[H] \cap \mT$. Denote by $T$ the tree on which  $\mathrm{N}_{G}[H]$ acts with stabilisers as such subgroups. Note that the action of $G$ on $\mathbb R$ via the natural projection of $G$ onto $\Z$ makes $\mathbb R$ into a model for $E_{\mT}G$.  Restricting this action to $\mathrm{N}_{G}[H]$ and considering the induced action on the join $T\ast \R$ gives us a $3$-dimensional model for  $E_{{\mT}[H]}\mathrm{N}_{G}[H]$. Invoking Proposition \ref{LW} entails  a $3$-dimensional model for $E_{\mH}G$, as was required to prove. 
\end{proof}

\begin{remark} Since finitely generated nilpotent groups lie ${{\scriptstyle\bf H}}_1^{\Fvc}\Fvc$, it follows that countable virtually nilpotent groups  are in ${{\scriptstyle\bf H}}_2^{\Fvc}\Fvc$. We obtain that the group $G=A\rtimes \Z \in {{\scriptstyle\bf H}}_3^{\Fvc}\Fvc$.
\end{remark}

\begin{remark} In the statement of Theorem \ref{cyclic ext}, one could enlarge $\mH$ to be the family of all locally virtually nilpotent subgroups of $G$. Then its proof together with Proposition \ref{P}(c)-(d) would imply that $\mathrm{N}_{G}[H]$ is  $\mT[H]$. So,  a point with the trivial action of $\mathrm{N}_{G}[H]$ would then be a model for $E_{{\mT}[H]}\mathrm{N}_{G}[H]$ for each $H\in \mathcal{S}$. Applying Proposition \ref{LW} would give us a $2$-dimensional model for $E_{\mH}G$.
\end{remark}
In the next example, we  illustrate that the family $\mH$ of all virtually nilpotent subgroups of $G$ can contain nilpotent subgroups of $G$ of arbitrarily high nilpotency class. 

\begin{example}  Consider the unrestricted wreath product $W=\Z\wr \Z$. Rewriting this group as a semi-direct product, we have that $W=A\rtimes \Z$ where $A=\prod_{i\in \Z} \Z$ and the standard infinite cyclic subgroup of $W$ is generated by $t$ and acts on $A$ by translations. Define  $G$ to be the subgroup of $W$ given by $G=P\rtimes \Z$. For each $k>0$, note that $P^k_1$ is the subgroup of $A$ of all $k$-periodic sequences of integers and hence $P^k_1\cong \Z^k$. Since $P_1=\cup_{k>0}P^k_1$, it is countable of infinite rank. Similarly,  one can argue that $P_2/P_1$ is countable of infinite rank and hence $P_2$ is also countable. Continuing in this manner, one obtains that  $P_i$ is countable for each $i>0$ and since $P$ is a countable union of these groups it is itself countable. This shows that the group $G$ satisfies the hypothesis of Theorem \ref{cyclic ext}.\\
Now, it is not difficult to see, that for each $i>0$, the subgroup $P^1_i\rtimes \Z$ of $G$ is nilpotent of  nilpotency class $i$.
\end{example}

\begin{theorem}\label{closure for VC} Let $\mF$ be a class of subgroups of finitely generated groups.  Then $\hff{}$ is closed under countable directed unions. If $\mF$ is the class of all virtually cyclic groups, then $\hff{}$ is closed under finite extensions and  under extensions with virtually soluble kernels. In particular, $\LHFF$ contains all locally virtually soluble groups.
\end{theorem}
\begin{proof}  The proof of the first fact is the same as for the class of finite groups $\mF$ given in Proposition 5.5 in \cite{DPT}. That is,  let $G$ be a countable directed union of groups that are in $\hff{}$. Then $G$ acts on a tree with stabilisers exactly the subgroups that comprise this  union.  It is   now easy to see that the action of $G$ on the tree satisfies the stabiliser and the fixed-point set conditions of 
the definition of $\hff{}$-groups. This shows that $G$ is in $\hff{}$.

For the second part, first note that by the Serre's Construction,  $\hff{}$ is closed under finite extensions (see the proof of \cite[2.3(2)]{lueck}). 

Let $G$ be a countable group that fits into an extension $K\rightarrowtail G \twoheadrightarrow Q$ such that $K$ is virtually soluble and $Q\in \hff{}$. Suppose $K$ is finite. Then  an easy transfinite induction on the ordinal associated to the class containing $Q$ shows $G$ lies in $\hff{}$. In general, since $K$ is virtually soluble, it contains a soluble characteristic subgroup of finite index, which must be normal in $G$. In view of these facts, without loss of generality, we can assume that $K$ is soluble.\\
Next, we proceed by the induction on the derived length of $K$ to prove that $G\in \hff{}$. 
When $K$ is the trivial group, then $G=Q\in \hff{}$.  Suppose  $K$ is nontrivial. Since $[K, K]$ is a characteristic subgroup of $K$, it is a normal subgroup of $G$. So, there are extensions $$[K,K]\rightarrowtail  G\twoheadrightarrow G/{[K,K]} \;\;\; \mbox{ and } \;\;\;
 K/{[K,K]}\rightarrowtail  G/{[K,K]}\twoheadrightarrow Q.$$ We claim that $G/{[K,K]}\in \hff{}$. Then by induction applied to the first extension $G\in \hff{}$. Let us now prove the claim. 

In view of the second extension, it suffices to show that given an extension   $$A\rightarrowtail S \twoheadrightarrow Q$$ where $A$ is abelian and  $Q\in \hff{_\alpha}$, then $S\in \hff{}$. We use  transfinite induction on the ordinal $\alpha$. When $\alpha= 0$, then $S$ is virtually a semi-direct product $A\rtimes \Z$. Hence, by Theorem \ref{cyclic ext}, it is in $\hff{}$. 

Suppose $\alpha > 0$, then there is a finite dimensional  $Q$-CW-complex $X$ such that each stabiliser subgroup lies in $\hff\beta$ for some $\beta<\alpha$ and such that $X^H$ is contractible for all $H \in \FF.$ The group $S$ also acts on $X$ via the projection onto $Q$. Each stabiliser of this action is abelian-by-$\hff\beta$ and hence by transfinite induction is in $\hff{}$. Therefore, $S\in \hff{}$. This finishes the claim and the proof.
\end{proof}

Recall that a subgroup $G$ of $\mathrm{GL}_n(\mathbb C)$ is said to be of {\it integral characteristic} if the coefficients of the characteristic polynomial of  every element of $G$ are algebraic integers. It follows that $G$ has integral characteristic if and only if the characteristic roots of every element of $G$ are algebraic integers (see \cite[\S 2]{AS}).

\begin{theorem}\label{thm: linear is HFF} Let $G$ be a countable subgroup of some $\mathrm{GL}_n(\mathbb C)$ of integral characteristic. Then $G$ lies in ${{\scriptstyle\bf H}}^{\Fvc}\Fvc.$ 
\end{theorem}
\begin{proof}  Since the class ${{\scriptstyle\bf H}}^{\Fvc}\Fvc$ is  closed under countable directed unions, it is enough to prove the claim when $G$ is finitely generated. Note that under a standard embedding of $\mathrm{GL}_n(\mathbb C)$  into $\mathrm{SL}_{n+1}(\mathbb C)$ the image of $G$ is still of integral characteristic. So, we can assume that $G$ is a subgroup of $\mathrm{SL}_{n}(\mathbb C)$  of integral characteristic. Let $A$ be the finitely generated subring of $\mathbb C$ generated by the matrix entries of a finite set of generators of $G$ and their inverses. Then $G$ is a subgroup of  $\mathrm{SL}_{n}(A)$. 

Let $\mathbb F$ denote the quotient field of $A$. Proceeding as in the proof of Theorem 3.3 of \cite{AS}, there is  an epimorphism $\rho: G \to H_1\times \dots \times H_r$ such that the kernel $U$ of $\rho$ is a unipotent subgroup of $G$ and for each $1\leq i\leq r$, $H_i$ is a subgroup of some 
$\mathrm{GL}_{n_i}(A)$ of integral characteristic where the canonical action of $H_i$ on $\mathbb F^{n_i}$ is irreducible and $\sum n_i=n$.  So, by  the proof of Theorem B in \cite{DKP}, each group $H_i$ admits a  finite dimensional model for $E_{\Fvc\cap H_i}H_i$. Applying \cite[5.6]{lueckweiermann}, one immediately sees that the  product  $Q=H_1\times \dots \times H_r$ admits a finite dimensional model for $E_{\Fvc\cap Q}Q$. So, $Q$ is in ${{\scriptstyle\bf H}}_1^{\Fvc}\Fvc$. By Theorem \ref{closure for VC}, it follows that $G$ lies in ${{\scriptstyle\bf H}}^{\Fvc}\Fvc$.
\end{proof}

\begin{corollary}\label{cor: linear has finite dim} Let $\mF$ be either the class of all finite groups or the class of all virtually cyclic groups and let $G$ be a group such that ${\widehat{\mathrm{H}}}^*_{\mF}(G,-)$ commutes with direct limits. If  $G$ is  a  subgroup of some $\mathrm{GL}_n(\mathbb C)$ of integral characteristic or if $G$ is  a  subgroup of some $\mathrm{GL}_n(\mathbb F)$ where $\mathbb F$ is a field of positive characteristic, then $\cd_{\mF}(G)<\infty$.
\end{corollary}
\begin{proof} Suppose $H$ is a finitely generated subgroup of $G$. If $G$ is  a subgroup of $\mathrm{GL}_n(\mathbb C)$ of integral characteristic, then by \cite{AS}  when $\mF$ is the class of finite groups or  or by the previous theorem when $\mF$ is the class of virtually cyclic groups, we know that $H$ in $\HFF$. If $G$ embeds into  $\mathrm{GL}_n(\mathbb F)$ for some field $\mathbb F$ of positive characteristic, then by \cite[Corollary 5]{DegPet4}, $H$ has finite Bredon cohomological dimension and hence it is in $\HFF$. This shows that $G$ is in $\lhff$. The result now follows from Theorem \ref{big}.
\end{proof}

\section{Change of family} 

In this section we discuss the question when the functor $\widehat\Ho^*_\FF(G,-)$ commutes with direct limits.  By the above, it is obvious that groups of finite Bredon cohomological dimension as well as groups of Bredon-type $\FP_\infty$ satisfy this condition.  It would be interesting to see whether there are groups a priori  satisfying neither, that also have continuous $\widehat\Ho^*_\FF(G,-)$. 

Considering Lemma \ref{tatects}, we see that it is enough to require that $\Ho^k_\FF(G,-)$ commutes with direct limits for infinitely many $k$. This, for example holds for groups, for which the trivial Bredon-module $\Z_\FF$ has a Bredon-projective resolution, which is finitely generated from a certain point onwards.

As mentioned in the introduction, the families of greatest interest are the families $\Ffin$ of finite subgroups and $\Fvc$ of virtually finite subgroups.  In light of Juan-Pineda and Leary's conjecture \cite{jpl}, which asserts that no non-virtually cyclic group is of type $\UUFP_\infty$, the question above is of particular interest for the family $\Fvc.$

Let us begin with the following:

\begin{question}\label{ctscomplete-triv-fin}
Does $\widehat\Ho^*(G,-)$ being continuous imply that $  \widehat\Ho^*_{\Ffin}(G,-)$ is continuous?
\end{question}

The converse of this question is obviously not true. Take any group $G$ with $\cd_{\Ffin}  G < \infty$, which is not of type $\FP_\infty$ and which has no bound on the orders of the finite subgroups.  It follows from \cite{kropholler} that groups with $\cd_\Q G <\infty$ and continuous $ \widehat\Ho^*(G,-)$ have a bound on the orders of their finite subgroups.  Locally finite groups and Houghton's groups satisfy this condition. On the other hand \cite[Theorem 2.7]{kropholler13}, any group $G$ in $\lhf$, for which $\widehat\Ho^*(G,-)$ is continuous has finite $cd_{\Ffin} G$, hence $\widehat\Ho^*_{\Ffin}(G,-)$ is continuous.

Also note that there are examples of groups of type $\FP_\infty$, which are not of type Bredon-$\FP_\infty$ for the class of finite subgroups \cite{ln03}. These groups, however, satisfy $\cd_{\Ffin}  G < \infty$, hence have continuous $  \widehat\Ho^*_{\Ffin}(G,-).$

\begin{question}
Is $\widehat\Ho^*_{\Ffin}(G,-)$ being continuous equivalent to $ \widehat\Ho^*_{\Fvc}(G,-)$ being continuous?
\end{question}

Any group of type $\UUFP_\infty$ is of type $\UFP_\infty$ (see \cite{desiconchabritaBLMS}) and any group with $\cd_{\Fvc} G <\infty$ also has $\cd_{\Ffin} G <\infty$ (see \cite{lueckweiermann}). Hence we may ask:

\begin{question}\label{vcyc-cts-qu}
Suppose $\cd_{\Ffin} G < \infty.$  Does this imply that $\widehat\Ho^*_{\Fvc}(G,-)$ is continuous?
\end{question}

If this question has a positive answer, Theorem A would imply that any group in ${{\scriptstyle\mathrm{L}\bf H}}^{\Fvc}\Fvc$ with $\cd_{\Ffin} G< \infty$ satisfies $\cd_{\Fvc} G< \infty.$\\ 

We end with two questions on the family ${{\scriptstyle\mathrm{L}\bf H}}^{\Fvc}\Fvc$.

\begin{question}\label{ext closer}
Is the class  ${{\scriptstyle\mathrm{L}\bf H}}^{\Fvc}\Fvc$ closed under extensions?
\end{question}

This reduces to asking whether an infinite cyclic extension of group in ${{\scriptstyle\mathrm{L}\bf H}}^{\Fvc}\Fvc$ is also in ${{\scriptstyle\mathrm{L}\bf H}}^{\Fvc}\Fvc$.

\begin{question}\label{elem amenable}
Does the class ${{\scriptstyle\mathrm{L}\bf H}}^{\Fvc}\Fvc$ contain all elementary amenable groups?
\end{question}

Note that a positive answer to Question \ref{ext closer} implies a positive answer to this question.

\end{document}